\documentclass[a4paper]{amsart}
\usepackage{fixltx2e}
\usepackage[utf8]{inputenc}
\usepackage[T1]{fontenc}
\usepackage{amssymb,amsthm,amsmath}
\usepackage[british]{babel}
\usepackage{xcolor}
\usepackage[all]{xy}
\usepackage{calrsfs}
\usepackage{graphicx}
\usepackage{hyperref}
\usepackage{mathbbol}
\usepackage{mathabx}

\theoremstyle{plain}
\newtheorem{theorem}[subsection]{Theorem}
\newtheorem{lemma}[subsection]{Lemma}
\newtheorem{proposition}[subsection]{Proposition}

\theoremstyle{definition}

\theoremstyle{remark}
\newtheorem{example}[subsection]{Example}
\newtheorem{remark}[subsection]{Remark}

\newenvironment{tfae}
{
\begin{enumerate}}
{\end{enumerate}}

\newcommand{\noproof}{\hfill \qed}

\newcommand{\comp}{\raisebox{0.2mm}{\ensuremath{\scriptstyle{\circ}}}}
\newcommand{\defn}{\textbf}
\newcommand{\To}{\Rightarrow}
\newcommand{\tensor}{\ensuremath{\otimes}}
\newcommand{\gnab}{\text{\rm \textexclamdown}}

\newcommand{\links}{\lgroup}
\newcommand{\rechts}{\rgroup}

\DeclareMathOperator{\ev}{ev}
\DeclareMathOperator{\Hom}{Hom}
\DeclareMathOperator{\kar}{char}

\newcommand{\C}{\ensuremath{\mathcal{C}}}
\newcommand{\V}{\ensuremath{\mathcal{V}}}

\newcommand{\K}{\ensuremath{\mathbb{K}}}
\newcommand{\N}{\ensuremath{\mathbb{N}}}
\newcommand{\Z}{\ensuremath{\mathbb{Z}}}

\newcommand{\coc}{\ensuremath{\mathsf{coc}}}
\newcommand{\AAAlg}{\ensuremath{\mathsf{AAAlg}}}

\newcommand{\Alt}{\ensuremath{\mathsf{ACAlg}}}
\newcommand{\BDAlg}{\ensuremath{\text{$B$-$\mathsf{DAlg}$}}}
\newcommand{\Lie}{\ensuremath{\mathsf{Lie}}}

\newcommand{\Alg}{\ensuremath{\mathsf{Alg}}}
\newcommand{\Vect}{\ensuremath{\mathsf{Vect}}}
\newcommand{\Cat}{\ensuremath{\mathsf{Cat}}}
\newcommand{\DRng}{\ensuremath{\mathsf{DRng}}}
\newcommand{\Gp}{\ensuremath{\mathsf{Gp}}}
\newcommand{\Leib}{\ensuremath{\mathsf{Leib}}}
\newcommand{\SLeib}{\ensuremath{\mathsf{SLeib}}}
\newcommand{\Mag}{\ensuremath{\mathsf{Mag}}}

\newcommand{\Pt}{\ensuremath{\mathsf{Pt}}}

\newcommand{\XMod}{\ensuremath{\mathsf{XMod}}}
\newcommand{\Set}{\ensuremath{\mathsf{Set}}}
\newcommand{\Hopf}{\ensuremath{\mathsf{Hopf}_{\K,\coc}}}
\newcommand{\CoAlg}{\ensuremath{\mathsf{CoAlg}_{\K,\coc}}}

\newcommand{\ACC}{{\rm (ACC)}}
\newcommand{\LACC}{{\rm (LACC)}}

\def\pullback{
 \ar@{-}[]+R+<6pt,-1pt>;[]+RD+<6pt,-6pt>%
 \ar@{-}[]+D+<1pt,-6pt>;[]+RD+<6pt,-6pt>}

\def\dottedpullback{%
 \ar@{.}[]+R+<6pt,-1pt>;[]+RD+<6pt,-6pt>%
 \ar@{.}[]+D+<1pt,-6pt>;[]+RD+<6pt,-6pt>}

\newdir{>}{{}*:(1,-.2)@^{>}*:(1,+.2)@_{>}}
\newdir{<}{{}*:(1,+.2)@^{<}*:(1,-.2)@_{<}}

\begin{document}

\title[A characterisation of Lie algebras]{A characterisation of Lie algebras\\ amongst anti-commutative algebras}

\author{Xabier García-Martínez}
\address[Xabier García-Martínez]{Departamento de Matemáticas, Esc.\ Sup.\ de Enx.\ Informática, Campus de Ourense, Universidade de Vigo, E--32004, Ourense, Spain\newline
and\newline
Faculty of Engineering, Vrije Universiteit Brussel, Pleinlaan 2, B--1050 Brussel, Belgium}
\email{xabier.garcia.martinez@uvigo.gal}

\author{Tim Van~der Linden}
\address[Tim Van~der Linden]{Institut de
Recherche en Math\'ematique et Physique, Universit\'e catholique
de Louvain, che\-min du cyclotron~2 bte~L7.01.02, B--1348
Louvain-la-Neuve, Belgium}
\thanks{}
\email{tim.vanderlinden@uclouvain.be}

\thanks{This work was partially supported by Ministerio de Economía y Competitividad (Spain), grant MTM2016-79661-P. The first author was also supported by Xunta de Galicia, 
	grant GRC2013-045 (European FEDER support included), 
	by an FPU scholarship of the Ministerio de Educación, Cultura y Deporte (Spain) and by a Fundaci\'on Barri\'e scholarship. The first author is a Postdoctoral Fellow of the Research Foundation--Flanders (FWO). The second author is a Research
Associate of the Fonds de la Recherche Scientifique--FNRS}

\begin{abstract}
Let $\K$ be an infinite field. We prove that if a variety of anti-commutative $\K$-algebras---not necessarily associative, where $xx=0$ is an identity---is locally algebraically cartesian closed, then it must be a variety of Lie algebras over $\K$. In particular, $\Lie_{\K}$ is the largest such. Thus, for a given variety of anti-commutative $\K$-algebras, the Jacobi identity becomes equivalent to a categorical condition: it is an identity in~$\V$ if and only if $\V$ is a subvariety of a locally algebraically cartesian closed variety of anti-commutative $\K$-algebras. This is based on a result saying that an algebraically coherent variety of anti-commutative $\K$-algebras is either a variety of Lie algebras or a variety of anti-associative algebras over $\K$.
\end{abstract}


\maketitle

\section{Introduction}
The aim of this article is to prove that, if a variety of anti-commutative algebras---not necessarily associative, where $xx=0$ is an identity---over an infinite field admits \emph{algebraic exponents} in the sense of James Gray's Ph.D.\ thesis~\cite{GrayPhD}, so when it is \emph{locally algebraically cartesian closed} (or \LACC\ for short, see~\cite{Gray2012,Bourn-Gray}), then it must necessarily be a variety of Lie algebras. 
Since, as shown in~\cite{GrayLie}, the category $\Lie_{\K}$ of Lie algebras over a commutative unitary ring $\K$ is always \LACC, this condition may be used to characterise Lie algebras amongst anti-commutative algebras. 

The only other non-abelian ``natural'' examples of locally algebraically cartesian closed semi-abelian~\cite{Janelidze-Marki-Tholen} categories we currently know of happen to be categories of group objects in a cartesian closed category~\cite{Gray2012}, namely
\begin{enumerate}
\item the category~$\Gp$ of groups itself;
\item the category $\XMod$ of crossed modules, which are the group objects in the category $\Cat$ of small categories~\cite{LR,MacLane}; and
\item the category $\Hopf$ of cocommutative Hopf algebras over a field $\K$ of characteristic zero~\cite{GKV, acc}, the group objects in the category $\CoAlg$ of cocommutative coalgebras over $\K$.
\end{enumerate}
At first with our project we hoped to remedy this situation by finding further examples of \LACC\ categories of (not necessarily associative) algebras. However, all of our attempts at constructing such new examples failed. Quite unexpectedly, in the end we managed to prove that, at least when the field $\K$ is infinite, amongst those algebras which are anti-commutative, \emph{there are no other examples}: the condition \LACC\ implies that the Jacobi identity holds. Thus, in the context of anti-commutative algebras, the Jacobi identity is characterised in terms of a purely categorical condition. This is the subject of Section~\ref{Main result}.

We do not study here what happens when the algebras considered are not anti-commutative. The category of Leibniz algebras is not \LACC, so at least one of the implications in our characterisation fails in that case. We make a few additional observations in Section~\ref{anti-commutative}, and leave the main question for the article~\cite{GM-VdL3}. 

\subsection{Cartesian closedness}
\emph{Algebraic exponentiation} is a categorical-algebraic version of the concept of \emph{exponentiation} familiar from set theory, linear algebra, topology, etc. In its most basic form, exponentiation amounts to the task of equipping the set $\Hom_{\C}(X,Y)$ of morphisms from $X$ to $Y$ with a suitable structure making it an object $Y^{X}$ in the category $\C$ at hand, while at the same time making the \emph{evaluation map} $\ev\colon{X\times Y^X\to Y}$ into a morphism. 

Depending on the given category $\C$, this may or may not be always possible. A~category with binary products $\C$ is said to be \defn{cartesian closed} when every object $X$ is \defn{exponentiable}, which means that the functor $X\times(-)\colon {\C\to \C}$ admits a right adjoint $(-)^{X}\colon {\C\to \C}$, so that for all $Y$ and $Z$ in~$\C$, the set ${\Hom_{\C}(X\times Z,Y)}$ is naturally isomorphic to ${\Hom_{\C}(Z,Y^{X})}$. This condition may be formulated in terms of a universal property as follows---see, for instance, \cite[Section~A1.5]{Johnstone:Elephant}: an object $X$ is exponentiable if and only if for every $Y$ there exists an object $Y^{X}$ and a morphism $\ev\colon{X\times Y^X\to Y}$ (called the \defn{evaluation}) such that for every $h\colon{X\times Z\to Y}$ there is a unique $\overline{h}\colon {Z\to Y^X}$ for which the triangle
\[
\xymatrix{X\times Z \ar[rr]^-{1_X\times \overline{h}} \ar[rd]_-{h} && X\times Y^X \ar[ld]^-{\ev} \\
& Y}
\]
commutes. 

The category $\Set$ of sets is cartesian closed, with $Y^{X}$ the set $\Hom_{\Set}(X,Y)$ of functions from $X$ to $Y$. The evaluation map $\ev\colon{X\times Y^X\to Y}$ takes a couple $(x,f)\in X\times Y^X$ and sends it to $f(x)\in Y$. Also the category $\Cat$ of small categories is cartesian closed. The category~$Y^{X}$ has functors ${X\to Y}$ as objects, and natural transformations between them as morphisms. For any commutative ring $\K$, the category $\CoAlg$ of cocommutative coalgebras over $\K$ is cartesian closed by a result in~\cite{Barr-Coalgebras}.

\subsection{Closedness in general}
The categories occurring in algebra are seldom cartesian closed. The concept of closedness has thus been extended in several different directions. One option is to replace the cartesian product by some other product, such as for instance the tensor product $\tensor_{\K}$ when $\C$ is the category $\Vect_{\K}$ of vector spaces over a field $\K$. In that case the result is the well-known tensor/hom adjunction, where the object $Y^{X}$ in the isomorphism $\Hom_{\K}(X\tensor_{\K} Z,Y)\cong\Hom_{\K}(Z,Y^{X})$ is the set of $\K$-linear maps $\Hom_{\K}(X,Y)$ with the pointwise $\K$-vector space structure.

\subsection{An alternative approach}
Another option, fruitful in non-abelian algebra, is to keep the cartesianness aspect of the condition, but to make it algebraic in an entirely different way~\cite{GrayPhD,Gray2012,Bourn-Gray}. To do this, we first need to understand what is \emph{local} cartesian closedness by reformulating the condition in terms of slice categories. Here we follow Section~A1.5 of~\cite{Johnstone:Elephant}.

\subsection{Bundles and their global sections}
Let $\C$ be any category. Given an object $B$ of~$\C$, we write $(\C\downarrow B)$ for the \defn{slice category} or \defn{category of bundles over~$B$} in which an object $x$ is an arrow $x\colon {X\to B}$ in~$\C$, and a morphism $f\colon{x\to y}$ is a commutative triangle 
\[
\xymatrix{X \ar[rr]^-{f} \ar[rd]_-{x} && Y \ar[ld]^-{y} \\
& B}
\]
in $\C$, so that $y\comp f=x$.

A \defn{global section} of a bundle $y\colon{Y\to B}$ is the same thing as a \defn{global element} of this object $y$: a morphism $1\to y$, where $1$ is the terminal object $1_B$ of $(\C\downarrow B)$. In other words, it is a section $f\colon{B\to Y}$ of $y\colon{Y\to B}$, so that $y\comp f = 1_B$.

\subsection{Local cartesian closedness}
Assuming now that $\C$ is finitely complete, given a morphism $a\colon {A\to B}$, we write 
\[
a^{*}\colon{(\C\downarrow B)\to (\C\downarrow A)}
\]
for the \defn{change-of-base functor} which takes an arrow $x\colon {X\to B}$ in $\C$ and sends it to its pullback $a^{*}(x)$ as in the diagram
\[
\xymatrix{A\times_{B}X \ar[r] \ar[d]_-{a^{*}(x)} \pullback & X \ar[d]^-{x} \\
A \ar[r]_-{a} & B.}
\]
If $B$ is the terminal object $1$ of $\C$ then $(\C\downarrow B)=(\C\downarrow 1)\cong \C$. Any object $A$ of $\C$ now induces a unique morphism $a=!_{A}\colon {A\to 1}$, and the functor $!^{*}_{A}\colon{\C\to (\C\downarrow A)}$ sends an object $Y$ to the product $A\times Y$ (considered together with its projection to~$A$). It is easily seen that the category $\C$ is cartesian closed if and only if for every $X$ in $\C$, the functor $!^{*}_{X}$ admits a right adjoint.

A category with finite limits $\C$ is said to be \defn{locally cartesian closed} or \defn{(LCC)} when for \emph{every} morphism $a\colon {A\to B}$ in $\C$ the change-of-base functor $a^{*}$ has a right adjoint. Equivalently, all slice categories $(\C\downarrow B)$ are cartesian closed---so that $\C$ is cartesian closed, \emph{locally over $B$}, for all $B$ in $\C$. This condition is stronger than cartesian closedness (the case $B=1$); examples include any Grothendieck topos, in particular the category of sets, while for instance~\cite{Conduche-Pullback} the category $\Cat$ is not (LCC), even though it is cartesian closed.

\subsection{Categories of points}
We may now mimic the concept of (local) cartesian closedness in such a way that it applies to \emph{global sections of bundles} instead of the bundles themselves. The idea is that, where \emph{slice categories} are very useful in non-algebraic settings, for certain applications in algebraic categories a similar role may be played by \emph{categories of points}. 

Let $\C$ be any category. Given an object $B$ of~$\C$, we write $\Pt_{B}(\C)$ for the \defn{category of points over $B$} in which an object $(x,s)$ is a split epimorphism 
$x\colon {X\to B}$ in $\C$, together with a chosen section $s\colon {B\to X}$, so that $x\comp s=1_{B}$. So a point is a bundle $x$ together with a global section $s$ of it. Given two points $(x\colon {X\to B},s\colon {B\to X})$ and $(y\colon {Y\to B},t\colon {B\to Y})$ over $B$, a morphism between them is an arrow $f\colon{X\to Y}$ in $\C$ satisfying $y\comp f=x$ and $f\comp s=t$.

Change of base is done as for slice categories: since sections are preserved, given any morphism $a\colon {A\to B}$ in a finitely complete category $\C$, we obtain a functor 
\[
a^{*}\colon{\Pt_{B}(\C)\to \Pt_{A}(\C)}.
\]

\subsection{Protomodular and semi-abelian categories} 
A finitely complete category $\C$ is said to be \defn{Bourn protomodular}~\cite{Bourn1991,Bourn2001,Borceux-Bourn} when each of the change-of-base functors $a^{*}\colon{\Pt_{B}(\C)\to \Pt_{A}(\C)}$ reflects isomorphisms. If $\C$ is a pointed category, then this condition may be reduced to the special case where $A$ is the zero object and $a=\gnab _{B}\colon{0\to B}$ is the unique morphism. The pullback functor $\gnab _{B}^{*}\colon{\Pt_{B}(\C)\to \C}$ then sends a split epimorphism to its kernel. Hence, protomodularity means that the \defn{Split Short Five Lemma} holds: suppose that in the commutative diagram
\[
\xymatrix{K \ar[r]^-{k} \ar[d]_-{g} & X \ar[d]_-{f} \ar@<.5ex>[r]^-{x} & B \ar@<.5ex>[l]^-{s} \ar@{=}[d]\\
L \ar[r]_-{l} & Y \ar@<.5ex>[r]^-{y} & B, \ar@<.5ex>[l]^-{t}}
\]
the morphism $k$ is the kernel of $x$ and $l$ is the kernel of $y$, while $f$ is a morphism of points $(x,s)\to (y,t)$; if now $g$ is an isomorphism, then $f$ is also an isomorphism.

A pointed protomodular category which is Barr exact and has finite coproducts is called a \defn{semi-abelian} category~\cite{Janelidze-Marki-Tholen}. This concept unifies earlier attempts (including, for instance,~\cite{Huq, Gerstenhaber, Orzech}) at providing a categorical framework that extends the context of abelian categories to encompass non-additive categories of algebraic structures such as groups, Lie algebras, loops, rings, etc. In this setting, the basic lemmas of homological algebra---the \emph{$3\times 3$ Lemma}, the \emph{Short Five Lemma}, the \emph{Snake Lemma}---hold~\cite{Bourn2001,Borceux-Bourn}, and may be used to study, say, (co)homology, radical theory, or commutator theory for those non-additive structures.

In a semi-abelian category, any point $(x,s)$ with its induced kernel $k$ as above gives rise to a split extension, since $x$ is also the cokernel of $k$, so that $(k,x)$ is a short exact sequence. By the results in~\cite{Bourn-Janelidze:Semidirect}, split extensions are equivalent to so-called \emph{internal actions} by means of a semi-direct product construction. Through this equivalence, there is a unique internal action $\xi\colon{B\flat K\to K}$ such that $X\cong K\rtimes_{\xi}B$. Without going into further details, let us just mention here that the object $B\flat K$ is the kernel of the morphism $\links 1_{B}\; 0\rechts\colon {B+K\to B}$, that the functor $B\flat(-)\colon{\C\to \C}$ is part of a monad, and that an internal $B$-action is an algebra for this monad. The category~$\Pt_{B}(\C)$ is monadic over $\C$, and its equivalence with the category of~$B\flat(-)$-algebras bears witness of this fact.

\subsection{Examples}
All \emph{Higgins varieties of $\Omega$-groups}~\cite{Higgins} are semi-abelian, which means that any pointed variety of universal algebras whose theory contains a group operation is an example. In particular, we find categories of all kinds of (not necessarily associative) algebras over a ring as examples, next to the categories of groups, crossed modules, and groups of a certain nilpotency or solvability class. Other examples include the categories of Heyting semilattices, loops, compact Hausdorff groups and the dual of the category of pointed sets~\cite{Janelidze-Marki-Tholen,Borceux-Bourn}.

\subsection{Algebraic cartesian closedness and the condition \LACC}
A category with finite limits $\C$ is said to be \defn{locally algebraically cartesian closed} or \defn{(LACC)} when for every morphism $a\colon {A\to B}$ in $\C$, the change-of-base functor $a^{*}\colon{\Pt_{B}(\C)\to \Pt_{A}(\C)}$ has a right adjoint~\cite{GrayPhD}. This condition is much stronger than \defn{algebraic cartesian closedness} or \defn{(ACC)} which is the case $B=1$.

When a semi-abelian category is (locally) algebraically cartesian closed, this has some interesting consequences~\cite{Gray2012,Bourn-Gray,acc}. For one thing, \ACC\ is equivalent to the condition that every monomorphism in $\C$ admits a centraliser. The property \LACC\ implies categorical-algebraic conditions such as \emph{peri-abelianness}~\cite{Bourn-Peri}, \emph{strong protomodularity}~\cite{B1}, the \emph{Smith is Huq} condition~\cite{MFVdL}, \emph{normality of Higgins commutators}~\cite{CGrayVdL1}, and \emph{algebraic coherence}. We come back to the latter condition (which implies all the others mentioned) in detail, in Subsection~\ref{(AC)} below.

The condition \ACC\ is relatively weak, and has all \emph{Orzech categories of interest}~\cite{Orzech} for examples. In comparison, \LACC\ is very strong: as mentioned above, we have groups, Lie algebras, crossed modules, and cocommutative Hopf algebras over a field of characteristic zero as ``natural'' semi-abelian examples, next to all abelian categories. An example of a slightly different kind---because it is non-pointed---is any category of groupoids with a fixed object of objects~\cite{Bourn-Gray}.

In what follows, we shall need the following characterisation of \LACC, valid in semi-abelian varieties of universal algebras. Instead of checking that all change-of-base functors $a^{*}\colon{\Pt_{B}(\V)\to \Pt_{A}(\V)}$ have a right adjoint, it suffices to check that \emph{some} change-of-base functors preserve binary sums. 

\begin{theorem}\label{(LACC) via flat sum}
For a semi-abelian variety of universal algebras $\V$, the following are equivalent:
\begin{tfae}
\item $\V$ is locally algebraically cartesian closed;
\item for all $B$ in $\V$, the pullback functor $\gnab_{B}^{*}\colon{\Pt_{B}(\V)\to \V}$ preserves all colimits;
\item for all $B$ in $\V$, the functor $\gnab_{B}^{*}$ preserves binary sums;
\item the canonical comparison $
\links B\flat \iota_{X}\; B\flat \iota_{Y}\rechts \colon B\flat X + B\flat Y \to B\flat (X+Y)$ is an isomorphism for all $B$, $X$ and $Y$ in $\V$.
\end{tfae}
\end{theorem}
\begin{proof}
This combines Theorem~2.9, Theorem~5.1 and Proposition 6.1 in~\cite{Gray2012}.
\end{proof}

Via the equivalence between split extensions and internal actions, condition (ii) means that the forgetful functor from the category of $B$-actions in $\V$ to $\V$ preserves all colimits. Hence in this varietal context, \LACC\ amounts to the property that colimits in the category of internal $B$-actions in $\V$ are independent of the acting object $B$, and computed in the base category $\V$.

\subsection{Algebraic coherence}\label{(AC)}
The concept of an \defn{algebraically coherent} category was introduced in~\cite{acc} with the aim in mind of having a condition with strong categorical-algebraic consequences such as the ones mentioned above for \LACC, while at the same time keeping all \emph{Orzech categories of interest} as examples. It is to \emph{coherence} in the sense of topos theory~\cite[Section~A1.4]{Johnstone:Elephant} what \emph{algebraic cartesian closedness} is to \emph{cartesian closedness}: a condition involving slice categories has been replaced by a condition in terms of categories of points.

The formal definition is that all change-of-base functors $a^{*}\colon{\Pt_{B}(\C)\to \Pt_{A}(\C)}$ preserve jointly strongly epimorphic pairs of arrows. This is clearly weaker than asking that the $a^{*}$ preserve all colimits. We shall only need the following characterisation, which is essentially Theorem~3.18 in~\cite{acc}: algebraic coherence is equivalent to the condition that for all $B$, $X$ and $Y$, the canonical comparison
\begin{equation*}
\links B\flat \iota_{X}\; B\flat \iota_{Y}\rechts \colon B\flat X + B\flat Y \to B\flat (X+Y)
\end{equation*}
from Theorem~\ref{(LACC) via flat sum} is a regular epimorphism. 

Algebraic coherence has somewhat better stability properties than \LACC. For instance, any subvariety of a semi-abelian algebraically coherent variety is still algebraically coherent. We shall come back to this in the next section.

Some examples of semi-abelian varieties which are \emph{not} algebraically coherent are the varieties of loops, Heyting semilattices, and non-associative algebras (the category $\Alg_{\K}$ defined below).

\section{Main result}\label{Main result}
The aim of this section is to prove Theorem~\ref{Theorem Alt then Lie}, which says that any \LACC\ variety of anti-commutative algebras over an infinite field $\K$ is a category of Lie algebras over $\K$. On the way we fully characterise algebraically coherent varieties of anti-commutative algebras (Theorem~\ref{Theorem AC Alt means Lie or AAAlg}). This is an application of a more general result telling us that a variety of $\K$-algebras is algebraically coherent if and only if it is an \emph{Orzech category of interest} (Theorem~\ref{Theorem AC iff Orzech}).

\subsection{Categories of algebras and their subvarieties}
Let $\K$ be a field. A \defn{(non-associative) algebra $A$ over $\K$} is a $\K$-vector space equipped with a bilinear operation $[-,-]\colon {A\times A\to A}$, so a linear map ${A\tensor A\to A}$. We use the notations $[x,y]=x\cdot y=xy$ depending on the situation at hand, always keeping in mind that the multiplication need not be associative. We write $\Alg_{\K}$ for the category of algebras over~$\K$ with product-preserving linear maps between them. It is a semi-abelian category which is not algebraically coherent. A \defn{subvariety} of~$\Alg_{\K}$ is any equationally defined class of algebras, considered as a full subcategory~$\V$ of~$\Alg_{\K}$. 

The category of \defn{associative algebras} over $\K$ is the subvariety of~$\Alg_{\K}$ satisfying $x(yz)=(xy)z$. 

The category $\Alt_{\K}$ of \defn{anti-commutative algebras} over $\K$ is the subvariety of $\Alg_{\K}$ satisfying $xx=0$. If the characteristic of the field $\K$ is different from $2$, then this is easily seen to be equivalent to the condition $xy=-yx$, whence the name ``anti-commutative''.

The category $\AAAlg_{\K}$ of \defn{anti-associative algebras} over $\K$ is the subvariety of~$\Alg_{\K}$ satisfying $x(yz)=-(xy)z$. 

The category $\Lie_{\K}$ of \defn{Lie algebras} over $\K$ is the subvariety of anti-commutative algebras satisfying the \defn{Jacobi identity} $x(yz)+z(xy)+ y(zx)=0$.

An algebra is \defn{abelian} when it satisfies $xy=0$. The subvariety of $\Alg_{\K}$ determined by the abelian algebras is isomorphic to the category $\Vect_{\K}$ of vector spaces over $\K$. An algebra $A$ is abelian if and only if $+\colon{A\times A\to A}$ is an algebra morphism, which makes $(A,+,0)$ an internal abelian group, so an \defn{abelian object} in the sense of~\cite{Borceux-Bourn}.

\subsection{Algebras over infinite fields}
We assume that the field $\K$ is infinite, so that we can use the following result (Theorem~\ref{Theorem Homogeneity}, which is Corollary~2 on page~8 of~\cite{Shestakov}). We first fix some terminology. For a given set $S$, a \defn{polynomial} with variables in $S$ is an element of the free $\K$-algebra on $S$. Recall that the left adjoint ${\Set\to \Alg_{\K}}$ factors as a composite of the \emph{free magma} functor $M\colon{\Set\to \Mag}$ with the \emph{magma algebra} functor $\K[-]\colon {\Mag\to \Alg_{\K}}$. The elements of $M(S)$ are non-associative words in the alphabet $S$, and the elements of $\K[M(S)]$, the polynomials, are $\K$-linear combinations of such words. A \defn{monomial} in $\K[M(S)]$ is any scalar multiple of an element of~$M(S)$. The \defn{type} of a monomial $\varphi(x_{1},\dots,x_{n})$ is the element $(k_{1},\dots,k_{n})\in \N^{n}$ where $k_{i}$ is the degree of $x_{i}$ in $\varphi(x_{1},\dots,x_{n})$. A polynomial is \defn{homogeneous} if its monomials are all of the same type. Any polynomial may thus be written as a sum of homogeneous polynomials, which are called its \defn{homogeneous components}.

\begin{theorem}\label{Theorem Homogeneity}\cite{Shestakov}
If $\V$ is a variety of algebras over an infinite field, then all of its identities are of the form $\phi(x_{1},\dots,x_{n})=0$, where $\phi(x_{1},\dots,x_{n})$ is a (non-associative) polynomial, each of whose homogeneous components $\psi(x_{i_{1}},\dots, x_{i_{m}})$ again gives rise to an identity $\psi(x_{i_{1}},\dots, x_{i_{m}})=0$ in $\V$. \noproof
\end{theorem}

\subsection{Description of $B\flat X$ in $\Alg_{\K}$}
Let $B$ and $X$ be free $\K$-algebras. Then the object $B\flat X$, being the kernel of the morphism $\links 1_{B}\;0\rechts\colon B+X\to B$, consists of those polynomials with variables in $B$ and in $X$ which can be written in a form where all of their monomials contain variables in $X$. For instance, given $b$, $b'\in B$ and $x\in X$, the expression $(b(xx))b'$ is allowed, but~$bb'$ is not.

\subsection{The reflection to a subvariety $\V$ of $\Alg_{\K}$}
Let $B$ and $X$ be free $\K$-algebras. We write $\overline{B}$ and $\overline{X}$ for their respective reflections into $\V$, which are free $\V$-algebras. These induce short exact sequences in $\Alg_{\K}$ such as 
\[
\xymatrix{0 \ar[r] & [X] \ar[r] & X \ar[r]^-{\eta_{X}} & \overline{X} \ar[r] & 0}
\]
where $\eta_{X}$ is the unit at $X$ of the reflection from $\Alg_{\K}$ to $\V$. We never write the right adjoint inclusion, but note that it preserves all limits. The kernel $[X]$ is a kind of relative commutator; its elements are precisely those polynomials $\phi(x_1,\dots,x_n)\in X$ where $x_1$, \dots, $x_n$ are in the set of generators of $X$ and $\phi(x_1,\dots,x_n)=0$ is an identity of~$\V$. Reflecting sums now, then taking kernels to the left, we obtain horizontal split exact sequences
\[
\xymatrix{0 \ar[r] & (B\flat X)\cap [B+X] \ar@{.>}[d] \dottedpullback
 \ar[r] & [B+X] \ar[d] \ar@<.5ex>[r] & [B] \ar[d] \ar[r] \ar@<.5ex>[l] & 0\\
0 \ar[r] & B\flat X \ar@{.>}[d]_-{\rho_{B,X}} \ar[r] & B+X \ar[d]_-{\eta_{B+X}} \ar@<.5ex>[r]^-{\links 1_{B}\;0\rechts} & B \ar[d]^-{\eta_{B}} \ar[r] \ar@<.5ex>[l] & 0\\
0 \ar[r] & \overline{B}\flat_{\V} \overline{X} \ar[r] & \overline{B}+_{\V}\overline{X} \ar@<.5ex>[r]^-{\links 1_{\overline{B}}\;0\rechts} & \overline{B} \ar[r] \ar@<.5ex>[l] & 0}
\] 
in $\Alg_{\K}$, where the sum $\overline{B}+_{\V} \overline{X}\cong \overline{B+X}$ is taken in $\V$. Using, for instance, the ${3\times 3}$~Lemma, it is not difficult to see that the induced dotted arrow $\rho_{B,X}$ is a surjective algebra homomorphism. In fact, the upper left square is a pullback, and we have three vertical short exact sequences. The one on the left allows us to view the elements of $\overline{B}\flat_{\V} \overline{X}$ as polynomials in $B\flat X$, modulo those identities which hold in~$\V$ that are expressible in $B\flat X$. In particular, an element $\phi(x_1,\dots,x_n,b_1,\dots, b_m)$ of~$B\flat X$ belongs to the top left intersection if and only if $\phi(x_1,\dots,x_n,b_1,\dots, b_m)=0$ is an identity of~$\V$. We freely use this interpretation in what follows, abusing terminology and notation by making no distinction between the equivalence class of polynomials that is an element in the quotient $\overline{B}\flat_{\V} \overline{X}$, and an element in $B\flat X$ which represents it.

\subsection{Subvarieties of $\Lie_{\K}$ need not be \LACC} 
Subvarieties of locally algebraically cartesian closed categories need no longer be such: we may take the variety of Lie algebras that satisfy $x(yz)=0$ as an example. 

\begin{proposition}\label{Proposition Nil_2}
Let $\V$ be a variety of non-associative algebras in which $x(yz)=0$ is an identity. If $\V$ is locally algebraically cartesian closed, then it is abelian.
\end{proposition}
\begin{proof}
Let $B$, $X$ and $Y$ be free algebras in $\V$, respectively generated by their elements $b$, $x$ and~$y$. Via the composite adjunction
\[
\xymatrix{\Set \ar@<1ex>[r]^-{\text{Free}} \ar@{}[r]|-{\perp} & \V \ar@<1ex>[r]^-{B+(-)} \ar@{}[r]|-{\perp} \ar@<1ex>[l]^-{\text{Forget}} & \Pt_{B}(\V),\ar@<1ex>[l]^-{\gnab^*_B}}
\]
the split epimorphisms
\[
(\links 1_{B}\;0\rechts\colon B+X\to B,\quad \iota_{B}\colon B\to B+X)
\]
and
\[
(\links 1_{B}\;0\rechts\colon B+Y\to B,\quad \iota_{B}\colon B\to B+Y)
\]
correspond to the free $B$-actions respectively generated by $x$ and $y$. Their sum in~$\Pt_{B}(\V)$ is
\[
(\links 1_{B}\;0\;0\rechts\colon B+X+Y\to B,\quad \iota_{B}\colon B\to B+X+Y).
\]
Applying the kernel functor, \LACC\ tells us that the canonical morphism
\begin{equation*}
\links B\flat \iota_{X}\; B\flat \iota_{Y}\rechts \colon B\flat X + B\flat Y \to B\flat (X+Y)
\end{equation*}
is an isomorphism (Theorem~\ref{(LACC) via flat sum}). When considering the sum $B\flat X + B\flat Y$, we write $b_{1}$ and $b_{2}$ for the generators of the two distinct copies of~$B$; then $\links B\flat \iota_{X}\; B\flat \iota_{Y}\rechts$ maps the $b_{i}$ to~$b$, sends $x$ to~$x$ and~$y$ to~$y$.

Now $x\in B\flat X$ and $yb_{2}\in B\flat Y$ are such that $x\cdot yb_{2}$ is sent to zero by the above isomorphism, since the identity $x(yz)=0$ holds in $\V$, so that $x\cdot yb$ is zero in $B+X+Y$, of which $B\flat(X+Y)$ is a subobject. As a consequence, $x\cdot yb_{2}$ is zero in the sum $B\flat X + B\flat Y$. Recall that $b_{2}\not\in B\flat Y$, so that $yb_{2}$ cannot be decomposed as a product of $y$ and $b_{2}$ in~$B\flat Y$; and it cannot be written as a product in which more than one $y$ or $b_{2}$ appears either, since by Theorem~\ref{Theorem Homogeneity} we may assume that all identities in $\V$ are homogeneous. Hence $yb_2$ is not a product, so that $x\cdot yb_{2}$ can only be zero if either $yb_{2}$ is zero in~$B\flat Y$, or $xz=0$ is an identity in $\V$. In the former case, $yb_{2}$ is zero in the sum $B+Y$, which is a free algebra on $\{b_{2},y\}$; then $yz=0$ is an identity in~$\V$. In either case, $\V$ is abelian.
\end{proof}

\subsection{(Anti-)associative algebras}
Essentially the same argument gives us two further examples, which we shall need later on:

\begin{proposition}\label{Proposition AAAlg}
If a variety of either associative or anti-associative algebras is locally algebraically cartesian closed, then it is abelian.
\end{proposition}
\begin{proof}
In the anti-associative case we have $x$, $-xb_{1}\in B\flat X$ and $b_{2}y$, $y\in B\flat Y$ such that $x\cdot b_{2}y$ and $-xb_{1}\cdot y$ are sent to the same element in $B\flat (X+Y)$ by the above isomorphism $\links B\flat \iota_{X}\; B\flat \iota_{Y}\rechts\colon B\flat X + B\flat Y \to B\flat (X+Y)$.

Similarly, in the associative case, we see that $xb_{1}\cdot y$ and $x\cdot b_{2}y$ are two distinct elements of the sum ${B\flat X+B\flat Y}$ which the morphism $\links B\flat \iota_{X}\; B\flat \iota_{Y}\rechts$ sends to one and the same element of ${B\flat(X+Y)}$.
\end{proof}

\begin{lemma}\label{Lemma anti-associative}
Any variety of anti-commutative $\K$-algebras that satisfies the identity $x(xy)=0$ is a subvariety of~$\AAAlg_{\K}$.
\end{lemma}
\begin{proof}
Taking $x=a+b$ and $y=c$ gives us
\[
0=(a+b)((a+b)c)=a(ac)+b(ac)+a(bc)+b(bc)=b(ac)+a(bc)
\]
so that $a(bc)=-b(ac)$. It follows that $(uv)w=-w(uv)=u(wv)=-u(vw)$ is an identity in $\V$, and $\V$ is a variety of anti-associative algebras.
\end{proof}

\subsection{Algebraic coherence}
Theorem~\ref{Theorem Homogeneity} gives us a characterisation of algebraic coherence for varieties of $\K$-algebras.

\begin{theorem}\label{Theorem AC iff Orzech}
Let $\K$ be an infinite field. If $\V$ is a variety of non-associative $\K$-algebras, then the following are equivalent:
\begin{tfae}
\item $\V$ is algebraically coherent;
\item there exist $\lambda_{1}$, \dots, $\lambda_{16}$ in $\K$ such that
\begin{align*}
z(xy)=
\lambda_{1}y(zx)&+\lambda_{2}x(yz)+
\lambda_{3}y(xz)+\lambda_{4}x(zy)\\
&+\lambda_{5}(zx)y+\lambda_{6}(yz)x+
\lambda_{7}(xz)y+\lambda_{8}(zy)x
\end{align*}
and 
\begin{align*}
	(xy)z=
	\lambda_{9}y(zx)&+\lambda_{10}x(yz)+
	\lambda_{11}y(xz)+\lambda_{12}x(zy)\\
	&+\lambda_{13}(zx)y+\lambda_{14}(yz)x+
	\lambda_{15}(xz)y+\lambda_{16}(zy)x
\end{align*}
are identities in $\V$;
\item $\V$ is a \defn{$2$-variety} in the sense of~\cite{Zwier}: for any ideal $I$ of an algebra $A$, the subalgebra $I^2$ of $A$ is again an ideal;
\item $\V$ is an \emph{Orzech category of interest}~\cite{Orzech}.
\end{tfae}
\end{theorem}
\begin{proof}
From the results of~\cite{acc} we already know that (iv) implies (i). It follows immediately from the definition of an \emph{Orzech category of interest} that (ii) implies~(iv). The equivalence between (ii) and (iii) is well known~\cite{Anderson}. To see that (i) implies (ii), we take free $B$-actions as in the first part of the proof of Proposition~\ref{Proposition Nil_2} and obtain the regular epimorphism 
\[
\links B\flat \iota_{X}\; B\flat \iota_{Y}\rechts \colon B\flat X + B\flat Y \to B\flat (X+Y).
\]
Any element $b(xy)$ of $B\flat (X+Y)$ is the image through this morphism of some polynomial $\psi(b_{1},x,b_{2},y)$ in $B\flat X + B\flat Y$. Note that this polynomial cannot contain any monomials obtained as a product of a $b_{i}$ with $xy$ or $yx$. This allows us to write, in the sum $B+X+Y$, the element $b(xy)$ as 
\begin{align*}
\lambda_{1}y(bx)&+\lambda_{2}x(yb)+
\lambda_{3}y(xb)+\lambda_{4}x(by)
+\lambda_{5}(bx)y+\lambda_{6}(yb)x+
\lambda_{7}(xb)y+\lambda_{8}(by)x\\
&+\nu\phi(b,x,y)
\end{align*}
for some $\lambda_{1}$, \dots, $\lambda_{8}$, $\nu\in \K$, where $\phi(b,x,y)$ is the part of the image of $\psi(b_{1},x,b_{2},y)$ in $B+X + Y$ which is not in the homogeneous component of $b(xy)$. Since $B+X+Y$ is the free $\V$-algebra on three generators $b$, $x$ and $y$, from Theorem~\ref{Theorem Homogeneity} we deduce that the first equation in (ii) is again an identity in $\V$. Analogously for $(xy)b$ we deduce the second equation in (ii).
\end{proof}

\begin{remark}
This result may be used to prove the claim made in~\cite{acc} that the category of \defn{Jordan algebras}---commutative and such that $(xy)(xx)=x(y(xx))$---is not algebraically coherent. Indeed, as explained in~\cite{Orzech}, it is not a \emph{category of interest}. 
\end{remark}

In the case of anti-commutative algebras, this characterisation becomes more precise:

\begin{theorem}\label{Theorem AC Alt means Lie or AAAlg}
Let $\K$ be an infinite field. If $\V$ is a subvariety of $\Alt_{\K}$, then the following are equivalent:
\begin{tfae}
\item $\V$ is algebraically coherent;
\item $\V$ is a subvariety of either $\AAAlg_{\K}$ or $\Lie_{\K}$.
\end{tfae}
\end{theorem}
\begin{proof}
(ii) implies (i) since $\AAAlg_{\K}$ and $\Lie_{\K}$ are \emph{Orzech categories of interest}~\cite{Orzech}, so their subvarieties are algebraically coherent. To see that (i) implies (ii), we first use anti-commutativity to simplify the identity given in Theorem~\ref{Theorem AC iff Orzech} to
\[
z(xy)=\lambda y(zx)+\mu x(yz)
\]
for some $\lambda$ and $\mu$ in $\K$. Choosing, in turn, $y=z$ and $x=z$, we see that 
\begin{enumerate}
\item either $\lambda=-1$ or $z\cdot zx=0$ is an identity in $\V$, and
\item either $\mu=-1$ or $x\cdot xy=0$ is an identity in $\V$.
\end{enumerate}
In any of the latter cases, $\V$ is a variety of anti-associative algebras by Lemma~\ref{Lemma anti-associative}. We are left with the situation when $\lambda=\mu=-1$, which means that the Jacobi identity holds in $\V$, so that $\V$ a variety of Lie algebras.
\end{proof}

\begin{example}
The variety of \emph{anti-commutative associative algebras} is an example. We have that $0 = x(yy) = (xy)y$ is an identity, so that by Lemma~\ref{Lemma anti-associative} those algebras are anti-associative as well. 

When $\kar(\K)\neq 2$, this implies that $xyz=0$ is an identity. We regain a variety as in Proposition~\ref{Proposition Nil_2}, so since it is not abelian, it cannot be \LACC.
\end{example}

\subsection{A characterisation of Lie algebras amongst anti-commutative algebras}
The condition \LACC\ eliminates one of the two options in Theorem~\ref{Theorem AC Alt means Lie or AAAlg}.

\begin{theorem}\label{Theorem Alt then Lie}
Let $\K$ be an infinite field. If $\V$ is a locally algebraically cartesian closed variety of anti-commutative $\K$-algebras, then it is a subvariety of~$\Lie_{\K}$. In other words, $\Lie_{\K}$ is the largest \LACC\ variety of anti-commutative $\K$-algebras. Thus for any variety $\V$ of anti-commutative $\K$-algebras, the following are equivalent:
\begin{tfae}
\item $\V$ is a subvariety of a \LACC\ variety of anti-commutative $\K$-algebras;
\item the Jacobi identity holds in $\V$.
\end{tfae}
\end{theorem}
\begin{proof}
This combines Theorem~\ref{Theorem AC Alt means Lie or AAAlg} with Proposition~\ref{Proposition AAAlg}.
\end{proof}

\begin{remark}
By Proposition~\ref{Proposition Nil_2}, the condition
\begin{tfae}
\setcounter{enumi}{2}
\item $\V$ is \LACC
\end{tfae}
is strictly stronger than the equivalent conditions (i) and (ii).
\end{remark}

\begin{remark}
We could not find any non-abelian examples of \LACC\ strict subvarieties of $\Lie_{\K}$. In the article~\cite{GM-VdL3}, it is explained why such varieties cannot exist.
\end{remark}

\section{Non-anti-commutative algebras}\label{anti-commutative}
An important question which we have to leave open for now, is what happens when the algebras we consider are not anti-commutative. We end this note with some of our preliminary findings, and study the question in detail in the article~\cite{GM-VdL3}.

\subsection{}
Some of the results and techniques used in the previous section are valid for non-anti-commutative algebras of course. For instance, Proposition~\ref{Proposition Nil_2}, Proposition~\ref{Proposition AAAlg} and Theorem~\ref{Theorem AC iff Orzech} are. 

\subsection{}
Proposition~\ref{Proposition Nil_2} tells us in particular that the variety of associative $\K$-algebras satisfying $xyz=0$ is not \LACC. 
Even though it is stated for infinite fields, this is still valid over the ring of integers $\Z$: we do not need to use Theorem~\ref{Theorem Homogeneity}, since we already know that all identities of this particular variety are linear combinations of homogeneous identities. This instance of the proposition contradicts Proposition~6.9 in~\cite{Gray2012}, which claims that the category $\DRng$ of all commutative non-unitary rings (=~$\Z$-algebras) satisfying $xyz=0$ is locally algebraically cartesian closed. Indeed, should this variety be \LACC, then it would be abelian---which is false. 

So, what is wrong? We noticed that the functor $R$ (which, up to equivalence, plays the role of the right adjoint of the functor $\gnab_B^*\colon \Pt_B(\DRng)\to \DRng$) constructed in the proof of~\cite[Proposition~6.9]{Gray2012} is not well defined on morphisms. Let us give a concrete example showing this in detail. We follow the notations from~\cite[Proposition~6.9]{Gray2012}: in particular, for any object $B$, the category $\BDAlg\simeq\Pt_B(\DRng)$ consists of $B$-modules equipped with a compatible commutative ring structure, which means that the identities 
\[
xyz=0,
\qquad
b(xy)=(bx)y=0
\qquad\text{and}\qquad
b(b'x)=(bb')x=0
\]
are satisfied. Then for $X$ in $\DRng$, the ring $R(X)$ is the set
\[
\{f\in \Hom(\Z\ltimes B,X)\mid \text{$f(0,b)x=0$ and $f(0,bb')=0$ for all $x\in X$, $b$, $b'\in B$}\}
\]
with multiplication $(fg)(n,b)=nf(1,0)g(1,0)$ and action $b'f(n,b)=f(0,nb')$ for $f$, $g\in R(X)$, $n\in \Z$ and $b$, $b'\in B$. 

Let $B=\langle b\mid bb=0\rangle=\Z[b]/(bb)$ act on the commutative ring
\[
X=\langle x,y\mid xx=xy=yy=0\rangle= \Z[x,y]/(xx,xy,yy)
\]
by $bx=y$ and $by=0$; then $X$ is not only an object of $\DRng$, but also a $B$-module. Consider also the object
\[
M=\langle m,p,q\mid mp=q,\; mq=pq=mm=pp=qq=0\rangle
\]
of $\DRng$. Let $h\colon X\to M$ be the ring homomorphism sending $x$ and $y$ to $m$. Let $f\in R(X)$ be defined by $f(n,b)=nx+bx$; note that $f=\eta_X(x)$, where $\eta\colon {1_{\BDAlg}\To RU}$ plays the role of the unit of the adjunction, and $U\colon {\BDAlg\to \DRng}$ is the forgetful functor. Then 
\[
R(h)(f)(0,b)\cdot p=(h\comp f)(0,b)\cdot p=h(y)\cdot p=mp=q\neq 0,
\]
which shows that $R(h)(f)$ is not an element of $R(M)$.

\subsection{}
The category $\Leib_{\K}$ of \defn{(right) Leibniz algebras}~\cite{Loday-Leibniz} over $\K$ is the subvariety of~$\Alg_{\K}$ satisfying the \defn{(right) Leibniz identity} $(xy)z=x(yz)+(xz)y$. This identity is clearly equivalent to the Jacobi identity when the algebras are anti-commutative, so that a Lie algebra is the same thing as an anti-commutative Leibniz algebra. However, examples of non-anti-commutative Leibniz algebras exist. Analogously, we can consider the category of \defn{(left) Leibniz algebras}, with corresponding identity $x(yz) = (xy)z + y(xz)$. Both categories are of course equivalent.

We do not know whether Theorem~\ref{Theorem Alt then Lie} extends to the non-anti-commutative case. What is certain, though, is that the category of Leibniz algebras is not locally algebraically cartesian closed. Indeed, using the notations of Proposition~\ref{Proposition Nil_2}, the Leibniz identity allows us to deduce
\[
\begin{cases}
(xy)b=x(yb)+(xb)y\\
(xy)b=-x(by)+(xb)y
\end{cases}
\]
so that $x\cdot yb = -x\cdot by$ in $B+X+Y$. This means that $x\cdot yb_{2}$ and $-x\cdot b_{2}y$ are two distinct elements of $B\flat X+B\flat Y$ which are sent to the same element of $B\flat(X+Y)$ by the morphism $\links B\flat \iota_{X}\; B\flat \iota_{Y}\rechts$. Hence this morphism cannot be an isomorphism, and $\Leib_{\K}$ is not \LACC.

We may ask ourselves what happens in the ``intersection'' between right and left Leibniz algebras. They are called \defn{symmetric Leibniz algebras} and, as shown in~\cite{MaYa}, the chain of inclusions $\Lie_{\K} \subseteq \SLeib_{\K} \subseteq \Leib_{\K}$ is strict. Doing a rearrangement of terms as in
\[
b(xy) = (bx)y + x(by) = b(xy) + (by)x + x(by),
\]
we see that $(by)x + x(by) = 0$. From this we may conclude that, in order to be (LACC), a variety of symmetric Leibniz algebras must either be anti-commutative or abelian. We thus regain the known cases of Lie algebras and vector spaces.

\subsection{}
A variety of algebras is anti-commutative precisely when the free algebra on a single generator is abelian: $xx=0$ is an identity in the variety, if and only if the bracket vanishes on the algebra freely generated by $\{x\}$. This corresponds to the condition that \emph{the free algebra on a single generator admits an internal abelian group structure}. This condition makes sense in arbitrary semi-abelian varieties, and we may ask ourselves whether perhaps it is implied by \LACC, as in the case of symmetric Leibniz algebras. This would allow us to drop the condition that $\V$ is anti-commutative in Theorem~\ref{Theorem Alt then Lie}.

The example of crossed modules proves that this is false. In~\cite{Carrasco-Homology} it is shown that on the one hand, a crossed module $\partial\colon{T\to G}$ with action $\xi$ admits an internal abelian group structure if and only if the groups $T$ and~$G$ are abelian and the action~$\xi$ is trivial. On the other hand, the free crossed module on a single generator is the inclusion $\kappa_{\Z,\Z}\colon\Z\flat \Z\to \Z+\Z$, equipped with the conjugation action. We see that in this case the free object on one generator is not abelian, even though $\XMod$ is a locally algebraically cartesian closed semi-abelian variety. However, it is not a variety of non-associative algebras of course.

\subsection{}
Perhaps this is not the right conceptualisation, and we must think of other ways of making the identity $xx=0$ categorical. The question then becomes whether \LACC, or any other appropriate catego\-rical-algebraic condition, would imply this new characterisation.

\section*{Acknowledgements}
Thanks to James R.~A.~Gray, George Janelidze, Zurab Janelidze for fruitful discussions and important comments on our work. We must thank the editor, Pino Rosolini, and the referee for very useful comments on our work that helped us improve the paper. We would also like to thank the University of Cape Town and Stellenbosch University for their kind hospitality during our stay in South Africa, and the Institut de Recherche en Mathématique et Physique (IRMP) for its kind hospitality during the first author's stays in Louvain-la-Neuve.


\providecommand{\noopsort}[1]{}
\providecommand{\bysame}{\leavevmode\hbox to3em{\hrulefill}\thinspace}
\providecommand{\MR}{\relax\ifhmode\unskip\space\fi MR }
\providecommand{\MRhref}[2]{%
  \href{http://www.ams.org/mathscinet-getitem?mr=#1}{#2}
}
\providecommand{\href}[2]{#2}

\end{document}